\RequirePackage{fix-cm}
\documentclass[smallextended]{svjour3}       
\smartqed  
\usepackage{latexsym,amsmath,amssymb,graphicx}
\usepackage{float}
\usepackage[bf, small]{caption}
\usepackage[all,web]{xy}
\usepackage{hyperref}

\DeclareFontFamily{OT1}{rsfs10}{}
\DeclareFontShape{OT1}{rsfs10}{m}{n}{ <-> rsfs10 }{}
\DeclareMathAlphabet{\mathscript}{OT1}{rsfs10}{m}{n}

\DeclareMathOperator{\Hom}{Hom}     
\DeclareMathOperator{\Tors}{Tors}    
\DeclareMathOperator{\Pic}{Pic}     
\DeclareMathOperator{\Cl}{Cl}       
\DeclareMathOperator{\rk}{rk}       
\DeclareMathOperator{\Mov}{Mov}     
\DeclareMathOperator{\Nef}{Nef}     
\DeclareMathOperator{\Eff}{Eff}     
\DeclareMathOperator{\codim}{codim} 
\DeclareMathOperator{\REF}{REF}     

\def \b{\beta }

\def \s{\sigma }

\def \Ga{\Gamma }
\def \Si{\Sigma }
\def \g{\gamma}
\def \aa{\mathbf{a}}

\def \q{\mathbf{q}}

\def \v{\mathbf{v}}

\def \w{\mathbf{w}}

\def \1{\mathbf{1}}
\def \0{\mathbf{0}}

\def\P{{\mathbb{P}}}

\def\p2{\mathbb{P}^2}
\def\p3{\mathbb{P}^3}
\def\p4{\mathbb{P}^4}

\def\rk{\operatorname{rk}}

\def\GL{\operatorname{GL}}

\def\Z{\mathbb{Z}}

\def\C{\mathbb{C}}
\def\R{\mathbb{R}}
\def\M{\mathbf{M}}
\def\Q{\mathbb{Q}}

\def\I{\mathcal{I}}
\def\Ls{\mathcal{L}}

\def\CQ{\mathcal{Q}}
\def\SF{\mathcal{SF}}

\def\gkz{\mathcal{Q}}

\def\G{\mathcal{G}}
\def\Ga{\Gamma}

\newcommand{\oneline}{\vskip12pt}
\newcommand{\halfline}{\vskip6pt}
\newcommand{\longmapsfrom}{\mathrel{\reflectbox{\ensuremath{\longmapsto}}}}

\begin{document}

\title{A numerical ampleness criterion via Gale duality\thanks{The authors were partially supported by the MIUR-PRIN 2010-11 Research Funds ``Geometria delle Variet\`{a} Algebriche''. The first author is also supported by the I.N.D.A.M. as a member of the G.N.S.A.G.A.}
}

\titlerunning{An ampleness criterion via Gale duality}        

\author{Michele Rossi \and Lea Terracini}

\authorrunning{M. Rossi and L.Terracini} 

\institute{M. Rossi and L. Terracini\at
              Dipartimento di Matematica, Universit\`a di Torino,
via Carlo Alberto 10, 10123 Torino \\
              Tel.: +39 011 670 2813\\
              Fax: +39 011 670 2878\\
              \email{michele.rossi@unito.it, lea.terracini@unito.it}           
           }

\date{}

\maketitle

\begin{abstract}
The main object of the present paper is a numerical criterion determining when a Weil divisor of a $\Q$--factorial complete toric variety admits a positive multiple Cartier divisor which is either numerically effective (nef) or ample. It is a consequence of $\Z$--linear interpretation of Gale duality and se\-con\-dary fan as developed in several previous papers of us. As a byproduct we get a computation of the Cartier index of a Weil divisor and a numerical characterization of weak $\Q$--Fano, $\Q$--Fano, Gorenstein, weak Fano and Fano toric varieties. Several examples are then given and studied.
\end{abstract}

\tableofcontents

\section*{Introduction}
This paper is devoted to present a numerical criterion, namely Theorem \ref{thm:criterio}, determining when a Weil divisor of a $\Q$--factorial complete toric variety admits a positive multiple Cartier divisor which is either numerically effective (nef) or ample. It is obtained as an easy consequence of $\Z$--linear interpretation of Gale duality and secondary fan developed in our previous papers \cite{RT-LA&GD}, \cite{RT-QUOT} and \cite{RT-Qfproj}. As an immediate application, we give a numerical characterization of (weak) $\Q$--Fano, (weak) Fano and Gorenstein toric varieties (see Theorem \ref{thm:Q-Fano}) and their toric coverings unramified in codimension $1$, called \emph{coverings in codimension $1$}, or simply \emph{$1$--coverings} (see Corollary \ref{cor:Q-Fano}).

Let us first of all recall that, in the literature, there are essentially three ways of presenting a $\Q$--factorial complete toric variety $X$ of dimension $n$:
\begin{enumerate}
  \item by giving a rational simplicial complete fan $\Si$ of cones in $N\otimes\R\cong\R^n$, with $N\cong\Z^n$: our hypothesis allows us to present $\Si$ by means of the $n$--skeleton $\Si(n)$ of its maximal cones and thinking of $\Si$ as the set of all the faces of cones in $\Si(n)$;
  \item by means of a Cox geometric quotient $(\C^{n+r}\setminus Z)/\Hom(\Cl(X),\C^*)$, along the lines of \cite{Cox}, where $r$ turns out to be the Picard number (in the following called \emph{rank}) of $X$;
  \item by means of a suitable \emph{bunch $\mathcal{B}$ of cones} in $F^r_{\R}=\Cl(X)\otimes\R$, as presented in \cite{Berchtold-Hausen}.
\end{enumerate}
In the first case (1) we will denote by $V$ the $n\times (n+r)$ integer \emph{reduced fan matrix} (see Definition \ref{def:matrice fan}) whose columns are given, up to the order, by the primitive generators of 1--dimensional cones (rays) in the 1--skeleton $\Si(1)$.

\noindent In the second case (2) the action of $G(X):=\Hom(\Cl(X),\C^*)$ can be described by means of a couple $(Q,\Ga)$ of matrices such that $Q$ is an integer and \emph{positive} (meaning that it admits only non--negative integer entries) $r\times (n+r)$ matrix in row echelon form ($\REF$, for short) and $\Ga$ is a \emph{torsion matrix}, as explained in \cite[Rem.~4.1]{RT-QUOT}. In the following $Q$ is called a \emph{weight matrix} of $X$, while we do not deal with the torsion matrix $\Ga$, whose definition is then skipped. In the simplest case of rank $r=1$, $X$ is always given by a finite abelian quotient of a weighted projective space (so called a \emph{fake} WPS) \cite{BC}, \cite{Conrads}, \cite{RT-WPS}, $Q$ gives the weight vector of the covering WPS and $\Ga$ describes the finite abelian quotient. For higher values of $r$, $X$ turns out to be a finite abelian quotient of a poly weighted space (PWS, for short) \cite[Thm.~2.2]{RT-QUOT}, $Q$ gives the weight matrix of the covering PWS and $\Ga$ describes the finite abelian quotient: notice that if $r\leq 2$ the $X$ turns out to be always a projective toric variety, but this fact does no more hold for $r\geq 3$ \cite{RT-R2proj}.

\noindent In the third case (3) the intersection of all the cones in the bunch $\mathcal{B}$ turns out to give the K\"{a}hler cone $\Nef(X)\subseteq F^r_{\R}$: in particular $X$ is projective if and only if $\Nef(X)$ is full dimensional and also called a \emph{chamber} of the secondary (or GKZ) fan.

\halfline

The ampleness criterion here proposed by Theorem \ref{thm:criterio} is completely based on the weight matrix $Q$ of $X$, differently from most part of known criteria, essentially based either on the fan $\Sigma$ or the polytope $\Delta_D$ associated with the divisor $D$. Then it is immediately applied to complete toric varieties presented as in either (2) or (3), above.

If we are dealing with a toric variety $X$ defined as in (1), then recall that a weight matrix of $X$ can be computed, starting from a fan matrix $V$, looking at the bottom $r$ rows of a switching matrix $U\in\GL(n+r,\Z)$ such that $H=U\cdot V^T$ gives the Hermite Normal Form (HNF) of the transposed fan matrix $V^T$ \cite[Prop.~4.3]{RT-LA&GD}: the advantage of this approach is that both $H$ and $U$ are computed by many computer algebra packages. Under our hypothesis we can always assume $Q$ to be a positive and REF matrix \cite[Thm.~3.18]{RT-LA&GD}.

A point of interest of the ampleness criterion presented in Theorem \ref{thm:criterio} is that checking either numerically effectiveness or ampleness of a toric divisor is completely reduced to linear algebraic operations on the weight matrix $Q$. We do not know if our approach may give rise to more efficient procedures than those which can been deduced by well known ampleness criteria on toric varieties, but certainly it looks easily understandable and quickly implementable in many computer algebra packages.

Let us notice that, from the purely theoretic point of view, Theorem \ref{thm:criterio} can be probably easily deduced by known results as, e.g., those given by Berchtold and Hausen in \cite{Berchtold-Hausen}. Anyway we believe that the linear algebraic approach here given is original and immediately applicable to toric data presenting a $\Q$-factorial complete toric variety.

\halfline

When this criterion is applied to the anticanonical Weil divisor $-K_X=\sum_jD_j$ one immediately gets a numerical criterion for detecting if a given $\Q$--factorial complete toric variety $X$ is either (weak) $\Q$--Fano or Gorenstein or even (weak) Fano: this is done in \S~\ref{sez:Q-Fano} and it also holds for every 1-covering of $X$. The conclusive \S~\ref{sez:esempi} is devoted to check when examples, whose geometric structure has been studied in \cite{RT-Qfproj},  give Gorenstein and (weak) $\Q$--Fano varieties, eventually up to a toric flip.

\section{Preliminaries and notation}
In the present paper we deal with $\Q$--factorial complete toric varieties associated with simplicial and complete fans. For preliminaries and used notation on toric varieties we refer the reader to \cite[\S~1.1]{RT-LA&GD}. We will also apply $\Z$--linear Gale duality as developed in \cite[\S~3]{RT-LA&GD}.
Every time the needed nomenclature will be recalled either directly by giving the necessary definition or by reporting the precise reference. Here is a list of main notation and relative references:

\subsection{List of notation}\label{ssez:lista}\hfill\\
Let $X(\Si)$ be a $n$--dimensional toric variety and $T\cong(\C^*)^n$ the acting torus, then
\begin{eqnarray*}
  &M,N,M_{\R},N_{\R}& \text{denote the \emph{group of characters} of $T$, its dual group and}\\
  && \text{their tensor products with $\R$, respectively;} \\
  &\Si\subseteq \frak{P}(N_{\R})& \text{is the fan defining $X$, where $\frak{P}(N_{\R})$ is the power set of $N_{\R}$} \\
  &\Si(i)& \text{is the \emph{$i$--skeleton} of $\Si$, which is the collection of all the}\\
  && \text{$i$--dimensional cones in $\Si$;} \\
  &|\Si|& \text{is the \emph{support} of the fan $\Si$, i.e.}\ |\Si|=\bigcup_{\s\in\Si}\s \subseteq N_{\R};\\
  &r=\rk(X)&\text{is the Picard number of $X$, also called the \emph{rank} of $X$};\\
  &F^r_{\R}& \cong\R^r,\ \text{is the $\R$--linear span of the free part of $\Cl(X(\Si))$};\\
  &F^r_+& \text{is the positive orthant of}\ F^r_{\R}\cong\R^r;\\
  &\langle\v_1,\ldots,\v_s\rangle\subseteq N_{\R}& \text{denotes the cone generated by the vectors $\v_1,\ldots,\v_s\in N_{\R}$;}\\
  && \text{if $s=1$ then this cone is also called the \emph{ray} generated by $\v_1$;} \\
  &\mathcal{L}(\v_1,\ldots,\v_s)\subseteq N& \text{denotes the sublattice spanned by $\v_1,\ldots,\v_s\in N$\,.}
\end{eqnarray*}
Let $A\in\mathbf{M}(d,m;\Z)$ be a $d\times m$ integer matrix, then
\begin{eqnarray*}
  &\mathcal{L}_r(A)\subseteq\Z^m& \text{denotes the sublattice spanned by the rows of $A$;} \\
  &\mathcal{L}_c(A)\subseteq\Z^d& \text{denotes the sublattice spanned by the columns of $A$;} \\
  &A_I\,,\,A^I& \text{$\forall\,I\subseteq\{1,\ldots,m\}$ the former is the submatrix of $A$ given by}\\
  && \text{the columns indexed by $I$ and the latter is the submatrix of}\\
  && \text{$A$ whose columns are indexed by the complementary }\\
  && \text{subset $\{1,\ldots,m\}\backslash I$;} \\
  &\REF& \text{Row Echelon Form of a matrix;}\\
  &\text{\emph{positive}}\ (\geq 0)& \text{a matrix (vector) whose entries are non-negative.}\\
  &\text{\emph{strictly positive}}\ (> 0)& \text{a matrix (vector) whose entries are strictly positive.}
\end{eqnarray*}
Given a $F$--matrix $V=(\v_1,\ldots,\v_{n+r})\in\mathbf{M}(n,n+r;\Z)$ (see Definition \ref{def:Fmatrice} below), then
\begin{eqnarray*}
  &\langle V\rangle& =\langle\v_1,\ldots,\v_{n+r}\rangle\subseteq N_{\R}\ \text{denotes the cone generated by the columns of $V$;} \\
  &\SF(V)& =\SF(\v_1,\ldots,\v_{n+r})\ \text{is the set of all rational simplicial fans $\Si$ such that}\\
  && \text{$\Sigma(1)=\{\langle\v_1\rangle,\ldots,\langle\v_{n+r}\rangle\}\subset N_{\R}$ \cite[Def.~1.3]{RT-LA&GD};}\\
  &\P\SF(V)&:=\{\Si\in\SF(V)\ |\ X(\Si)\ \text{is projective}\};\\
  &\G(V)&=Q\ \text{is a \emph{Gale dual} matrix of $V$ (see either Definition \ref{def:Galeduale} or \cite[\S~2.1]{RT-LA&GD})} \\
  &\CQ&=\langle\G(V)\rangle \subseteq F^r_+\ \text{is a \emph{Gale dual cone} of}\ \langle V\rangle:\ \text{it is always assumed to be}\\
  && \text{generated in $F^r_{\R}$ by the columns of a positive $\REF$ matrix $Q=\G(V)$}\\
  && \text{(see \cite[Thms.~2.8,~2.19]{RT-LA&GD})}.
\end{eqnarray*}
\oneline
\noindent Let us start by recalling the $\Z$--linear interpretation of Gale duality and the se\-con\-da\-ry fan following the lines of \cite[\S~2]{RT-LA&GD} and \cite[\S~1.2]{RT-Qfproj}, to which the interested reader is referred for any further detail.

\begin{definition}\label{def:Galeduale} Let $V$ be a $n\times (n+r)$ integer matrix of rank $n$. If we think $V$ as a linear application from $\Z^{n+ r}$ to $\Z^n$ then $\ker(V)$ is a lattice in $\Z^{n+ r}$, of rank $r$ and without cotorsion. We shall denote ${\mathcal G}(V)$ the \emph{Gale dual matrix of $A$}, which is an integral $r\times (n+ r)$ matrix $Q$ such that ${\mathcal L}_r(Q)=\ker(V)$.
\end{definition}
A Gale dual matrix $Q=\G(V)$ is well-defined up to left multiplication by $\GL_r(\Z)$.
Notice that ${\mathcal L}_r(Q)={\mathcal L}_r(V)^\perp$ w.r.t. the standard inner product in $\R^{n+r}$, and that $Q\cdot V^T=0$; moreover $Q$ can be characterized by the following
universal property:
\begin{equation*}
\forall\,A\in \mathbf{M}(r,r;\Z)\quad A\cdot V^T=0\ \Rightarrow\ \exists\,\alpha\in\mathbf{M}(r,r;\Z)\,:\,A=\alpha\cdot Q
\end{equation*}

\begin{definition}\label{def:matrice fan} A $n$--dimensional $\Q$--factorial complete toric variety $X=X(\Si)$ of rank $r$ is the toric variety defined by a $n$--dimensional \emph{simplicial} and \emph{complete} fan $\Si$ such that $\sharp(\Si(1))=n+r$ \cite[\S~1.1.2]{RT-LA&GD}. In particular the rank $r$ coincides with the Picard number i.e. $r=\rk(\Pic(X))$. A matrix $V\in\M(n,n+r;\Z)$, whose columns are given by non zero elements of the monoid $\rho\cap N$, one for every $\rho\in\Si(1)$, is called a \emph{fan matrix of $X$}. If every such element is actually a generator of the associated monoid $\rho\cap N$, i.e. it is a \emph{primitive} element of the ray $\rho$, then $V$ is called a \emph{reduced} fan matrix.
\end{definition}

\begin{definition}\cite[Def.~3.10]{RT-LA&GD}\label{def:Fmatrice} An \emph{$F$--matrix} is a $n\times (n+r)$ matrix  $V$ with integer entries, satisfying the conditions:
\begin{itemize}
\item[a)] $\rk(V)=n$;
\item[b)] $V$ is \emph{$F$--complete} i.e. $\langle V\rangle=N_{\R}\cong\R^n$ \cite[Def.~3.4]{RT-LA&GD};
\item[c)] all the columns of $V$ are non zero;
\item[d)] if ${\bf  v}$ is a column of $V$, then $V$ does not contain another column of the form $\lambda  {\bf  v}$ where $\lambda>0$ is real number.
\end{itemize}
A \emph{$CF$--matrix} is a $F$-matrix satisfying the further requirement
\begin{itemize}
\item[e)] the sublattice ${\mathcal L}_c(V)\subset\Z^n$ is cotorsion free, which is ${\mathcal L}_c(V)=\Z^n$ or, equivalently, ${\mathcal L}_r(V)\subset\Z^{n+r}$ is cotorsion free.
\end{itemize}
A $F$--matrix $V$ is called \emph{reduced} if every column of $V$ is composed by coprime entries \cite[Def.~3.13]{RT-LA&GD}. A (reduced) fan matrix of a $\Q$--factorial complete toric variety $X(\Si)$ is always a (reduced) $F$--matrix.
\end{definition}

\begin{definition}\cite[Def.~3.9]{RT-LA&GD}\label{def:Wmatrix} A \emph{$W$--matrix} is an $r\times (n+r)$ matrix $Q$  with integer entries, satisfying the following conditions:
\begin{itemize}
\item[a)] $\rk(Q)=r$;
\item[b)] ${\mathcal L}_r(Q)$ has not cotorsion in $\Z^{n+r}$;
\item[c)] $Q$ is \emph{$W$--positive}, which is $\mathcal{L}_r(Q)$ admits a basis consisting of positive vectors (see list \ref{ssez:lista} and \cite[Def.~3.4]{RT-LA&GD}).
\item[d)] Every column of $Q$ is non-zero.
\item[e)] ${\mathcal L}_r(Q)$   does not contain vectors of the form $(0,\ldots,0,1,0,\ldots,0)$.
\item[f)]  ${\mathcal L}_r(Q)$ does not contain vectors of the form $(0,a,0,\ldots,0,b,0,\ldots,0)$, with $ab<0$.
\end{itemize}
A $W$--matrix is called \emph{reduced} if $V=\G(Q)$ is a reduced $F$--matrix \cite[Def.~3.14, Thm.~3.15]{RT-LA&GD}.

\noindent The Gale dual matrix $Q=\G(V)$ of a fan matrix $V$ of a $\Q$--factorial complete toric variety $X(\Si)$ is a $W$--matrix called  a \emph{weight matrix of $X$}.
\end{definition}

\begin{definition}\cite[Def.~2.7]{RT-LA&GD}\label{def:PWS} A \emph{poly weighted space} (PWS) is a $n$--dimensional $\Q$--factorial complete toric variety $Y(\widehat{\Si})$ of rank $r$, whose reduced fan matrix is a $CF$--matrix i.e. if
\begin{itemize}
  \item $\widehat{V}$ is a $n\times(n+r)$ $CF$--matrix,
  \item $\widehat{\Si}\in\SF(\widehat{V})$.
\end{itemize}
Let us recall that a $\Q$--factorial complete toric variety $Y$ is a PWS if and only if it is \emph{1-connected in codimension 1} \cite[\S~3]{Buczynska}, which is $\pi^1_1(Y)\cong\Tors(\Cl(Y))=0$ \cite[Thm.~2.1]{RT-QUOT}.
\end{definition}

\begin{definition}\label{def:GKZ&Mov}\cite[Def.~1.6]{RT-Qfproj} Let $\mathcal{S}_r$ be the family of all the $r$--dimensional subcones of $\CQ$ obtained as intersection of simplicial subcones of $\CQ$. Then define the \emph{secondary fan} (or \emph{GKZ decomposition}) of $V$ to be the set $\Ga=\Ga(V)$ of cones in $F^r_+$ such that
\begin{itemize}
  \item its subset of $r$--dimensional cones (the \emph{$r$--skeleton}) $\Ga(r)$ is composed by the minimal elements, with respect to the inclusion, of the family $\mathcal{S}_r$,
  \item its subset of $i$--dimensional cones (the \emph{$i$--skeleton}) $\Ga(i)$ is composed by all the $i$--dimensional faces of cones in $\Ga(r)$, for every $1\leq i\leq r-1$.
\end{itemize}
 A maximal cone $\g\in\Ga(r)$ is called a \emph{chamber} of the secondary fan $\Ga$. Finally define
\begin{equation}\label{Mov}
    \Mov(V):= \bigcap_{i=1}^{n+r}\left\langle Q^{\{i\}}\right\rangle\ ,
\end{equation}
where $\left\langle Q^{\{i\}}\right\rangle$ is the cone generated in $F^r_+$ by the columns of the submatrix $Q^{\{i\}}$ of $Q$ (see the list of notation \ref{ssez:lista}).
\end{definition}

The fact that $\Ga$ is actually a fan is a consequence of the following

\begin{theorem}\label{thm:GKZ}
If $V$ is a $F$--matrix then, for every $\Si\in\SF(V)$,
\begin{enumerate}
  \item $\CQ=\overline{\Eff}(X(\Si))$, the \emph{pseudo--effective cone of $X$}, which is the closure of the cone generated by effective Cartier divisor classes of $X$ \cite[Lemma~15.1.8]{CLS},
  \item $\Mov(V)=\overline{\Mov}(X(\Si))$, the closure of the cone generated by movable Cartier divisor classes of $X$ \cite[(15.1.5), (15.1.7), Thm.~15.1.10, Prop.~15.2.4]{CLS}.
  \item $\Ga(V)$ is the secondary fan (or GKZ decomposition) of $X(\Si)$ \cite[\S~15.2]{CLS}. In particular $\Ga$ is a fan and $|\Ga|=\CQ\subset F^r_+$.
\end{enumerate}
\end{theorem}

\begin{theorem}\emph{\cite[Prop. 15.2.1]{CLS}}\label{thm:Gale sui coni} There is a one to one correspondence between the following sets
\begin{eqnarray*}
    \mathcal{A}_{\Ga}(V)&:=&\{\g\in\Ga(r)\ |\ \g\subset\Mov(V)\}\\
    \P\SF(V)&:=&\{\Si\in\SF(V)\ |\ X(\Si)\ \text{is projective}\}\ .
\end{eqnarray*}
\end{theorem}

For the following it is useful to understand the construction of such a correspondence. Namely (compare \cite[Prop.~15.2.1]{CLS}):
\begin{itemize}
  \item after \cite{Berchtold-Hausen}, given a chamber $\g\in\mathcal{A}_{\Ga}$ let us call \emph{the bunch of cones of $\g$} the collection of cones in $F^r_+$ given by
      \begin{equation*}
        \mathcal{B}(\g):=\left\{\left\langle Q_J\right\rangle\ |\ J\subset\{1,\ldots,n+r\}, \sharp(J)=r, \det\left(Q_J\right)\neq 0, \g\subset\left\langle Q_J\right\rangle\right\}
      \end{equation*}
      (see also \cite[p.~738]{CLS}),
  \item it turns out that $\bigcap_{\b\in\mathcal{B}(\g)}\b=\g$\,,
  \item  for any $\g\in\mathcal{A}_{\Ga}(V)$ there exists a unique fan $\Si_{\g}\in\P\SF(V)$ such that
  \begin{equation*}
    \Si_{\g}(n):=\left\{\left\langle V^J\right\rangle\ |\ \left\langle Q_J\right\rangle\in \mathcal{B}(\g)\right\}\,,
  \end{equation*}
  \item for any $\Si\in\P\SF(V)$ the collection of cones
  \begin{equation}\label{bunch}
    \mathcal{B}_{\Si}:=\left\{\left\langle Q^I\right\rangle\ |\ \left\langle V_I\right\rangle\in\Si(n)\right\}
  \end{equation}
  is the bunch of cones of the chamber $\g_{\Si}\in\mathcal{A}_{\Ga}$ given by $\g_{\Si}:=\bigcap_{\b\in\mathcal{B}_{\Si}}\b$.
\end{itemize}
Then the correspondence in Theorem \ref{thm:Gale sui coni} is realized by setting
\begin{equation*}
\begin{array}{ccc}
  \mathcal{A}_{\Ga}(V) & \longleftrightarrow & \P\SF(V) \\
  \g & \longmapsto & \Si_{\g} \\
  \g_{\Si} & \longmapsfrom & \Si
\end{array}
\end{equation*}

Let us here recall the following standard notation for a normal complete variety $X$:
\begin{itemize}
  \item a Cartier divisor $D\subset X$ is called \emph{ample} if it admits a \emph{very ample} positive multiple $kD$ that is, the associated morphism $$\varphi_{kD}:X\rightarrow\P\left(H^0(X,\mathcal{O}_X(D))^{\vee}\right)$$
      is a closed embedding;
  \item a Cartier divisor $D\subset X$ is called \emph{numerically effective} (\emph{nef}) if $D\cdot C\geq 0$ for every irreducible complete curve $C\subset X$
  \item $\Nef(X)$ is the cone generated by nef divisors
\end{itemize}
\begin{remark}\label{rem:ampio}In the toric setting nef divisors are especially easy to understand: in particular, if $X$ is a $\Q$-factorial complete toric variety then being nef for $D$ is equivalent to one of the following: (a) $D$ is basepoint free, (b) $\mathcal{O}_X(D)$ is generated by global sections, (c) $D\cdot C\geq 0$ for every torus-invariant, irreducible, complete curve $C\subset X$ \cite[Thm.~6.3.12]{CLS}.
\end{remark}

As a final result let us recall the following
\begin{proposition}\emph{\cite[Thm. 15.1.10(c)]{CLS}}\label{prop:nef} If $V=(\v_1\ldots,\v_{n+r})$ is a $F$--matrix then, for every fan $\Si\in\P\SF(V)$, there is a na\-tu\-ral isomorphism $\Pic(X(\Si))\otimes\R\cong F^r_{\R}$ taking the cones
\begin{equation*}
    \Nef(X(\Si))\subseteq\overline{\Mov}(X(\Si))\subseteq\overline{\Eff}(X(\Si))
\end{equation*}
to the cones
\begin{equation*}
    \g_{\Si}\subseteq\Mov(V)\subseteq\CQ\,.
\end{equation*}
In particular, calling $d:\mathcal{W}_T(X(\Si))\to\Cl(X(\Si))$  the morphism giving to a torus invariant divisor $D$ its linear equivalence class $d(D)$, we get that a Weil divisor $D$ on $X(\Si)$ admits a nef (ample) positive multiple if and only if $d(D)\in\g_{\Si}$ (\,$d(D)$ is an interior point of $\g_{\Si}$).
 \end{proposition}

 \begin{remark}\label{rem:nef} Let us notice that, for every fan $\Si\in\SF(V)$, the set of cones $\mathcal{B}_{\Si}$ defined in (\ref{bunch}), still gives a bunch of cones in the sense of Berchtold and Hausen \cite{Berchtold-Hausen}. Moreover, under our hypothesis, one still has that
 \begin{equation*}
   \Nef(X(\Si))=\bigcap_{\beta\in\mathcal{B}_{\Si}} \beta
 \end{equation*}
 (see either \cite[Thm.~9.2]{Berchtold-Hausen} or \cite[Prop.~15.1.3 (b), Prop.~15.2.1 (c)]{CLS}). \\
 In particular $\Si\in\P\SF(V)$ if and only $\Nef(X(\Si))$ is a full dimensional subcone of $\CQ=\overline{\Eff}(X(\Si))$.
 \end{remark}

\section{A numerical ampleness criterion}
Let $V=(\v_1,\ldots,\v_{n+r})$ be a $n\times(n+r)$ reduced $F$--matrix and $Q=\G(V)=(\q_1,\ldots,\q_{n+r})$ be a Gale dual REF positive $W$--matrix. Let $\Si\in\SF(V)$ and consider the $\Q$--factorial complete toric variety $X(\Si)$. Recall that, calling $\mathcal{W}_T(X)$ the $\Z$--module of torus invariant Weil divisors, the transposed matrix $V^T$ and the Gale dual matrix $Q=\G(V)$ are representative matrices of the morphisms giving the following standard exact sequence
 \begin{equation}\label{div-sequence}
    \xymatrix{0 \ar[r] & M\otimes\Q\ar[r]^-{div}_-{V^T}& \mathcal{W}_T(X)\otimes\Q
\ar[r]^-{d}_-Q & \Cl(X)\otimes\Q \ar[r]& 0}
\end{equation}
with respect to the standard basis of $\mathcal{W}_T(X)$ given by the torus invariant divisors $\{D_j\}_{j=1}^{n+r}$, where $D_j$ is the closure of the torus orbit of the ray $\langle\v_j\rangle$. Under the identification $\mathcal{W}_T(X)=\Z^{n+r}$ a divisor $D=\sum_j a_jD_j$ will be identified with the column vector $\aa=\left(a_1,\ldots,a_{n+r}\right)^T$. Consequently the free part of the linear equivalence class $d(D)\in F^r\subseteq\Cl(X)$ will be identified with the column $r$-vector $Q\cdot\aa$. The least positive integer $k\geq 1$ such that $kD$ is Cartier will be called the \emph{Cartier index of $D$} and denoted by $c(D)$.

Then setting
\begin{equation*}
  \mathcal{I}_\Si=\{I\subset\{1,\ldots,n+r\}\,|\,\left\langle V^I\right\rangle\in\Si(n)\}=\{I\subset\{1,\ldots,n+r\}\,|\,\left\langle Q_I\right\rangle\in\mathcal{B}(\g_{\Si})\}
\end{equation*}
one gets the following numerical ampleness criterion

\begin{theorem}\label{thm:criterio} Let $X(\Si)$ be a $\Q$-factorial complete toric variety. A Weil divisor $D=\sum_j a_jD_j\in\mathcal{W}_T(X)$ admits a positive multiple which is a nef (resp. ample) divisor if and only if for every $I\in\mathcal{I}_{\Si}$ the column vector
        \begin{equation}\label{ampiezza}
          \aa_I:=Q_I^{-1}\cdot d(D) =Q_I^{-1}\cdot Q\cdot \aa \geq 0\quad  (\text{resp.}\ >0)
       \end{equation}
i.e. it has positive (strictly positive) entries.

\noindent In particular, if $\Si\in\P\SF(V)$ and the chamber $\g_{\Si}\subseteq\Mov(V)$ is a $r$--di\-men\-sio\-nal simplicial cone, then its generators determine a maximal rank $r\times r$ matrix $G$ and condition $(\ref{ampiezza})$ admits the following shortcut:
\begin{equation}\label{ampiezza2}
  \aa_{\Si}=G^{-1}\cdot d(D) =G^{-1}\cdot Q\cdot \aa \geq 0\quad(\text{resp.}\ >0)\,.
\end{equation}
Moreover the Cartier index of $D$ is given by the least common multiple of denominators in either $\aa_I$, for every $I\in\I_{\Si}$, or $\aa_{\Si}$, respectively. In particular $D$ is a nef (resp. ample) divisor if and only if either $(\ref{ampiezza})$ or $(\ref{ampiezza2})$ hold and either $\aa_I$, for every $I\in\I_{\Si}$, or $\aa_{\Si}$, respectively, have integer entries.
\end{theorem}

\begin{proof} Recall Proposition \ref{prop:nef} and Remark \ref{rem:nef}. In particular,  $\Nef(X(\Si))\subseteq\gkz$ and, if $\Si\in\P\SF(V)$ then $\Nef(X(\Si))$ is the chamber $\g_{\Si}$. Since $X$ is $\Q$--factorial, a suitable positive multiple of $D$ is a Cartier divisor. It is a nef divisor if and only if $Q\cdot\aa\in \Nef(X(\Si))$ and it is an ample divisor if and only if $Q\cdot\aa\in\g_{\Si}\setminus \partial \g_{\Si}$. Notice that $$\Nef(X(\Si))=\bigcap_{\b\in\mathcal{B}_{\Si}}\b=\bigcap_{I\in\mathcal{I}_{\Si}}\langle Q_I\rangle$$
and observe that the $i$--th row of the inverse matrix $Q_I^{-1}$ gives a (rational) inward normal vector to the facet of the simplicial cone $\langle Q_I\rangle$ determined by the $r-1$ columns of $Q_I$ different from the $i$--th one. The same argument holds for the matrix $G$ when it is a $r\times r$ maximal rank matrix, which is when the chamber $\g_{\Si}$ is a simplicial cone.

For the last part of the statement, recall that \cite[Thm.\,3.2(2)]{RT-QUOT}
$$\Pic(X(\Si))=\bigcap_{I\in\I_{\Si}}\Ls_c(Q_I)\,.$$
To ending up the proof, notice that $Q\cdot\aa\in\Ls_c(Q_I)$ if and only if $\aa_I$ has integer entries.
\end{proof}

\begin{remark} Let us recall that in the toric setting a weaker version of the Kleiman ampleness criterion is given by condition (c) in Remark \ref{rem:ampio} with strict inequalities (see e.g. \cite[Prop.~3.7]{Fujino}). In a sense the previous Theorem \ref{thm:criterio} is the Gale dual version of such a weaker Kleiman criterion.
\end{remark}

\section{A numerical characterization of (weak) $\Q$--Fano toric varieties}\label{sez:Q-Fano}
As an application of the previous Theorem \ref{thm:criterio} one easily gets a quite effective numerical characterization of Gorenstein and (weak) $\Q$--Fano toric varieties.

In this context a $\Q$--Fano toric variety is a $\Q$--factorial complete toric variety $X(\Si)$ whose (Weil) anti-canonical divisor $-K_X=\sum_{j=1}^{n+r} D_j$ admits an ample po\-si\-ti\-ve multiple $-mK_X$. Hence $X$ is projective. If $m=1$, i.e. $-K_X$ is ample, $X(\Si)$ is a Fano variety.

\noindent If $-K_X$ admits a nef positive multiple $-mK_X$ then $X$ is called a \emph{weak} $\Q$--Fano toric variety. As above, if $m=1$, i.e. $-K_X$ is nef, then $X$ is called a weak Fano toric variety. Clearly a weak $\Q$--Fano toric variety is  complete but not necessarily projective.

\noindent In particular $X$ is Gorenstein if $K_X$ is Cartier.

Let $V$ be a $n\times (n+r)$ reduced fan matrix of $X(\Si)$, i.e. $V$ is a reduced $F$--matrix and $\Si\in\P\SF(V)$. Let $Q=\G(V)$ be a Gale dual, positive, $\REF$ $W$--matrix of $X$.

\begin{theorem}\label{thm:Q-Fano}
  In the above notation the following are equivalent:
  \begin{enumerate}
    \item $X(\Si)$ is a $\Q$--Fano (resp. weak $\Q$--Fano) toric variety,
    \item for every $I\in\mathcal{I}_{\Si}$ the column vector
        \begin{equation*}
          \w_I:=Q_I^{-1}\cdot d(-K_X) =Q_I^{-1}\cdot Q\cdot \left(
              \begin{array}{c}
                1 \\
              \vdots \\
                1 \\
              \end{array}
            \right)>0\quad (\text{resp.}\ \geq 0)
       \end{equation*}
i.e. it has strictly positive (resp. non-negative) entries. In particular, if the chamber $\g_{\Si}$ is a $r$--dimensional simplicial cone whose generators determine the maximal rank $r\times r$ matrix $G$ then the previous condition admits the following shortcut:
\begin{equation*}
  \w_{\Si}=G^{-1}\cdot d(-K_X) =G^{-1}\cdot Q\cdot \left(
              \begin{array}{c}
                1 \\
              \vdots \\
                1 \\
              \end{array}
            \right)>0\quad (\text{resp.}\ \geq 0)
\end{equation*}
\end{enumerate}
Moreover $X$ is Gorenstein if and only if either $\w_I$, for every $I\in\I_{\Si}$, or $\w_{\Si}$, respectively, have integer entries and $X$ is Fano (resp. weak Fano) if and only if, in addition, also the previous conditions hold.
\end{theorem}

 A finite surjective morphism $\varphi:Y\rightarrow X$ of $\Q$--factorial complete toric varieties is called a \emph{covering in codimension $1$} (or simply a \emph{$1$--covering}) if it is unramified in codimesion $1$, which is there exists a subvariety $V\subset X$ such that $\codim_X V\geq 2$ and $\varphi|_{Y_V}$ is a topological covering, where $Y_V:=\varphi^{-1}(X\setminus V)$. Moreover a \emph{universal covering in codimension $1$} is a $1$-covering $\varphi:Y\rightarrow X$ such that for any $1$--covering $\phi:X'\rightarrow X$ of $X$ there exists a $1$--covering $f:Y\rightarrow X'$ such that $\varphi=\phi\circ f$ \cite[Def~3.13, Rem.~3.14]{Buczynska}\cite[Def.~1.8]{RT-QUOT}. After \cite[Thm.~2.2]{RT-QUOT} every $\Q$--factorial complete toric variety $X$ admits a canonical universal 1--covering which is a PWS.

\begin{corollary}\label{cor:Q-Fano}
  Let $X(\Si)$ be a $\Q$--factorial complete toric variety and $Y(\widehat{\Si})$ be the PWS giving the 1--connected universal covering of $X$, as in \cite[Thm.~2.2]{RT-QUOT}. Then $X$ is (weak) $\Q$--Fano, Gorenstein, (weak) Fano if and only if $Y$ is (weak) $\Q$--Fano, Gorenstein, (weak) Fano, respectively.

  \noindent Moreover $X$ is (weak) $\Q$--Fano, Gorenstein, (weak) Fano if and only if every 1--covering $X'$ of $X$ is (weak) $\Q$--Fano, Gorenstein, (weak) Fano, respectively.
\end{corollary}

\begin{proof}
If $V$ is a fan matrix of $X$, then $\widehat{V}=\G(Q)=\G(\G(V))$ is a $CF$--matrix giving a fan matrix of $Y$ \cite[Prop.~3.11]{RT-LA&GD}. Therefore both $X$ and $Y$ admit a same weight matrix given by $\G(V)=Q=\G(\widehat{V})$ \cite[Def.~3.9]{RT-LA&GD}. Since the choice of $\Si\in\SF(V)$ uniquely determines the choice of $\widehat{\Si}\in\SF(\widehat{V})$, then $\mathcal{I}_{\Si}=\mathcal{I}_{\widehat{\Si}}$. This suffices to show that the ampleness conditions in part (2) of Theorem \ref{thm:Q-Fano} is the same both for $-K_X$ and $-K_Y$, ending up the proof of the first part of the statement. The second part follows immediately by replacing $X$ with the 1--covering $X'$.
\end{proof}

Let us observe that a similar statement does no more hold for any abelian co\-ve\-ring $Y\twoheadrightarrow X$: e.g. it does not hold for a toric cover, in the sense of \cite{AP} and \cite[\S~3.2.2]{RT-Qfproj}, since the weight matrices of the basis $X$ and of the covering $Y$ are no more the same. A counterexample is given by $X(\Si_{10})\stackrel{72:1}{\longrightarrow}\P^W(\mathcal{E})$ in Example~\ref{ex:WPTB(c)} below.

\section{Examples}\label{sez:esempi}
We conclude this paper by giving concrete applications of theorems \ref{thm:criterio} and \ref{thm:Q-Fano}. At this purpose we analyze all the examples presented in \cite{RT-Qfproj}.

  \begin{example}\label{ex:noconverse}\cite[Ex.~3.31]{RT-Qfproj}
    Let $X$ be the 2--dimensional PWS of rank 2 given by the blow up in one point of the weighted projective plane $\P(1,2,1)$, whose reduced weight and fan matrices are, respectively, given by
\begin{equation*}
    Q=\left(
        \begin{array}{cccc}
          1 & 2 & 1 & 0 \\
          0 & 1 & 1 & 1 \\
        \end{array}
      \right)\quad\Rightarrow\quad V=\G(Q)=\left(
                                             \begin{array}{cccc}
                                               1 & 0 & -1 & 1 \\
                                               0 & 1 & -2 & 1 \\
                                             \end{array}
                                           \right)
\end{equation*}
Then $\Mov(V)=\langle\q_2,\q_3\rangle=\left\langle\begin{array}{cc}
                            2 & 1 \\
                            1 & 1
                          \end{array}
\right\rangle \subset \gkz=F^2_+$, and there is a unique chamber $\g=\Mov(V)$, giving a unique fan $\Si_{\g}$.
Therefore
$$\w_{\Si_{\g}}=(\q_2,\q_3)^{-1}\cdot Q\cdot  \left(
              \begin{array}{c}
                1 \\
              \vdots \\
                1 \\
              \end{array}
            \right)= \left(
                       \begin{array}{cc}
                         1 & -1 \\
                         -1 & 2 \\
                       \end{array}
                     \right)\cdot\left(
                                   \begin{array}{c}
                                     4 \\
                                     3 \\
                                   \end{array}
                                 \right)=\left(
                                   \begin{array}{c}
                                     1 \\
                                     2 \\
                                   \end{array}
                                 \right)>0
            $$
            and $X$ is a Fano toric variety.
            \end{example}

\begin{example}\label{ex:PTB}\cite[Ex.~3.40]{RT-Qfproj} Consider the smooth projective toric variety $X(\Si)$ given by the blow up of $\P^3$ in two distinct points. X has dimension and rank equal to 3, with reduced fan matrix $V$ given by
\begin{equation*}
    V=\left(
        \begin{array}{cccccc}
          1 & 0 & 0 & 0 & -1 & 1 \\
          0 & 1 & 0 & 0 & -1 & 1 \\
          0 & 0 & 1 & -1 & -1 & 1 \\
        \end{array}
      \right)\ \Longrightarrow\ Q=\left(
                             \begin{array}{cccccc}
                               1 & 1 & 1 & 0 & 1 & 0 \\
                               0 & 0 & 1 & 1 & 0 & 0 \\
                               0 & 0 & 0 & 0 & 1 & 1 \\
                             \end{array}
                           \right)=\G(V)\,.
\end{equation*}
\begin{figure}
\begin{center}
\includegraphics[width=7truecm]{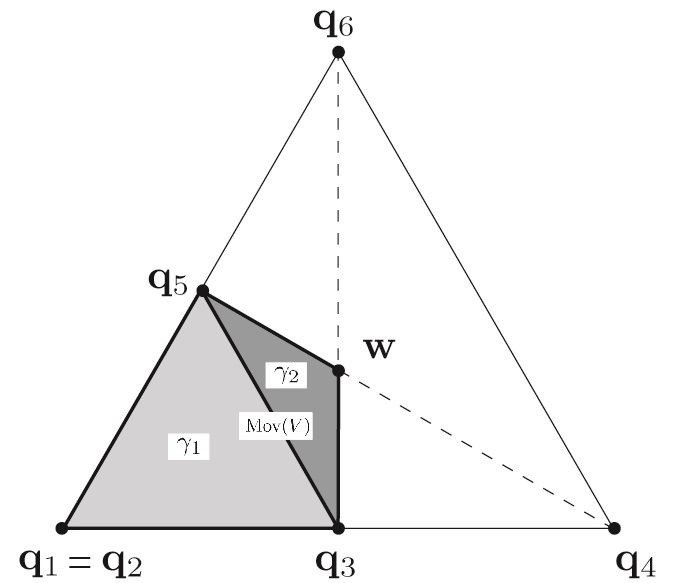}
\caption{\label{fig2}Ex.~\ref{ex:PTB}: the section of the cone $\Mov(V)$ and its chambers, inside the Gale dual cone $\mathcal{Q}=F^3_+$, as cut out by the plane $\sum_{i=1}^3x_i=1$.}
\end{center}
\end{figure}
One can visualize the Gale dual cone $\gkz=\langle Q\rangle =F^3_+$ and $\Mov(V)\subseteq\gkz$ giving a picture of their section with the hyperplane $\{x_1+x_2+x_3=1\}\subseteq F^3_{\R}\cong\Cl(X)\otimes\R$, as in Fig.~\ref{fig2}.
\cite{Kleinschmidt-Sturmfels} gives that $\P\SF(V)=\SF(V)$ and the latter is described by the only two
chambers of $\Mov(V)$ represented in Fig.~\ref{fig2}, so giving
$\P\SF(V)=\SF(V)=\{\Si_1=\Si_{\g_1},\Si_2=\Si_{\g_2}\}$.

\noindent Notice that
$$-d(K_X)=Q\cdot\left(
              \begin{array}{c}
                1 \\
              \vdots \\
                1 \\
              \end{array}
            \right)=\left(
                     \begin{array}{c}
                       4 \\
                       2 \\
                       2 \\
                     \end{array}
                   \right)=2\q_3+2\q_5$$
                   meaning that the anti-canonical class gives a point lying on the wall separating the only two chambers $\g_1,\g_2$ of $\Mov(V)$. Therefore for both the fans $\Si_1,\Si_2\in\SF(V)$, the anti-canonical divisor turns out to be a nef and no ample divisor: hence no fans in $\SF(V)$ give a $\Q$-Fano toric variety. Anyway $X(\Si_i)$ turns out to be Gorenstein and weak Fano, for both $i=1,2$. This fact can be deduced also by observing that both the chambers are simplicial and
                   $$\w_{\Si_1}=\left(
                                 \begin{array}{c}
                                   0 \\
                                   2 \\
                                   2 \\
                                 \end{array}
                               \right)\geq 0\quad,\quad\w_{\Si_2}=\left(
                                             \begin{array}{c}
                                               2 \\
                                               0 \\
                                               2 \\
                                             \end{array}
                                           \right)\geq 0\,.$$
\end{example}

\begin{example}\label{ex:nototmaxbord}\cite[Ex.~3.41]{RT-Qfproj} By adding the further column $\left(
                                                \begin{array}{c}
                                                  1 \\
                                                  1 \\
                                                  1 \\
                                                \end{array}
                                              \right)
$ in the weight matrix of the previous Example \ref{ex:PTB}, one gets the following reduced weight and fan matrices
\begin{equation*}
Q=\left(
                             \begin{array}{ccccccc}
                               1 & 1 & 1 & 0 & 1 & 1 & 0 \\
                               0 & 0 & 1 & 1 & 1 & 0 & 0 \\
                               0 & 0 & 0 & 0 & 1 & 1 & 1 \\
                             \end{array}
                           \right)\ \Rightarrow\ V=\left(
        \begin{array}{ccccccc}
          1 & 0 & 0 & 0 & 0 & -1 & 1 \\
          0 & 1 & 0 & 0 & 0 & -1 & 1 \\
          0 & 0 & 1 & 0 & -1 & 0 & 1 \\
          0 & 0 & 0 & 1 & -1 & 1 & 0
        \end{array}
      \right)=\G(Q)
\end{equation*}
The new weight column introduces a further subdivision in $\gkz=F^3_+$, leaving unchanged $\Mov(V)$ and giving
            $\P\SF(V)=\SF(V)=\{\Si_1,\Si_2,\Si_3,\Si_4\}$, as represented in Fig.~\ref{fig3}.
            \begin{figure}
            \begin{center}
            \includegraphics[width=7.5truecm]{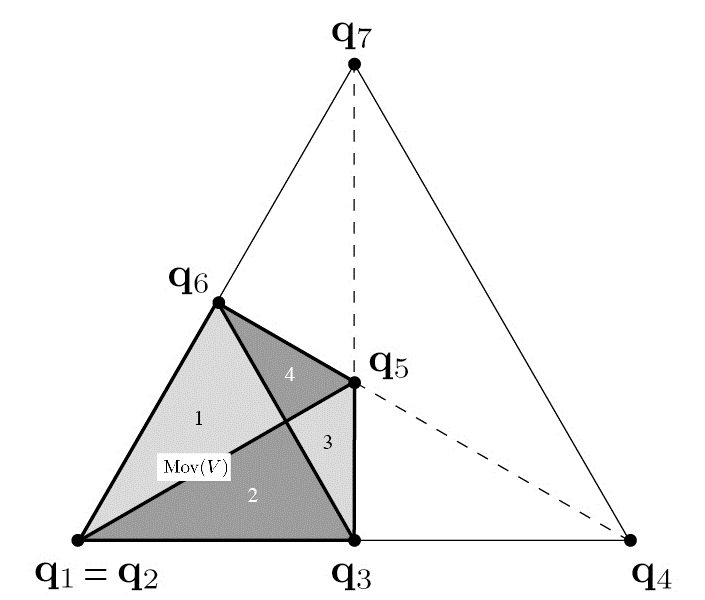}
            \caption{\label{fig3}Ex.~\ref{ex:nototmaxbord}: the section of the cone $\Mov(V)$ and its four chambers, inside the Gale dual cone $\mathcal{Q}=F^3_+$, as cut out by the plane $\sum_{i=1}^3x_i=1$.}
            \end{center}
            \end{figure}
             Notice that $X(\Si_i)$ is smooth for every $i=1,\ldots,4$.

\noindent Also in this case one can easily observe that
    $$-d(K_X)=Q\cdot\left(
              \begin{array}{c}
                1 \\
              \vdots \\
                1 \\
              \end{array}
            \right)=\left(
                     \begin{array}{c}
                       5 \\
                       3 \\
                       3 \\
                     \end{array}
                   \right)=2\q_3+2\q_6+\q_5 = 2\q_1+3\q_5$$
                   meaning that, on the one hand, the anti-canonical class gives an internal point of the cone $\langle\q_3,\q_5,\q_6\rangle=\g_3+\g_4$ and, on the other hand, it gives a  point lying on the wall passing through $\q_1$ and $\q_5$ (see Fig.~\ref{fig3}). Therefore $-d(K_X)$ is a point on the wall separating chambers $\g_3$ and $\g_4$ and the anti-canonical divisor $-K_X$ turns out to be a nef no ample divisor for $X(\Si_3)$ and $X(\Si_4)$. Moreover $-K_X$ is not even a nef divisor for $X(\Si_1)$ and $X(\Si_2)$. Then no fans in $\SF(V)$ give a $\Q$-Fano toric variety, but $X(\Si_3)$ and $X(\Si_4)$ are weak Fano. Such a conclusion can be numerically obtained by observing that all the four chambers are simplicial and
                   $$\w_{\Si_1}=\left(
                                                    \begin{array}{c}
                                                      3 \\
                                                      0 \\
                                                      -1 \\
                                                    \end{array}
                                                  \right)\ ,\ \w_{\Si_2}=\left(
                                                                \begin{array}{c}
                                                                  0 \\
                                                                  3 \\
                                                                  -1 \\
                                                                \end{array}
                                                              \right)\ ,
                                                              \ \w_{\Si_3}=
                                                              \left(\begin{array}{c}
                                                              2 \\
                                                              1 \\
                                                              0 \\
                                                              \end{array}
                                                              \right)\geq 0\ ,
                                                              \ \w_{\Si_4}=\left(
                                                              \begin{array}{c}
                                                              1 \\
                                                              2 \\
                                                              0 \\
                                                              \end{array}
                                                              \right)\geq 0\,.$$
Clearly $X(\Si_i)$ is Gorenstein for every $1\leq i\leq 4$.
\end{example}

\begin{example}\label{ex:noWPTB}\cite[Ex.~3.42]{RT-Qfproj} Let $X$ be the 2--dimensional PWS of rank 3 given by the blow up of the weighted projective plane $\P(1,2,1)$ in two distinct points, whose reduced fan and weight matrices are given by
\begin{equation*}
    V=\left(
        \begin{array}{ccccc}
          1 & 0 & -1 & 1 & -1 \\
          0 & 1 & -2 & 1 & -1 \\
        \end{array}
      \right)\ \Longrightarrow\ Q=\left(
                             \begin{array}{ccccc}
          1 & 1 & 0 & 0 & 1 \\
          0 & 1 & 1 & 1 & 0 \\
          0 & 0 & 0 & 1 & 1 \\
        \end{array}
                           \right)=\G(V)\,.
\end{equation*}
Then $\P\SF(V)=\SF(V)=\{\Si\}$ and the unique simplicial complete fan $\Si$ is as\-so\-cia\-ted with the unique chamber $\g=\Mov(V)\subset\mathcal{Q}$
which is not a simplicial cone (see Fig.\ref{fig4}).
\begin{figure}
\begin{center}
\includegraphics[width=8truecm]{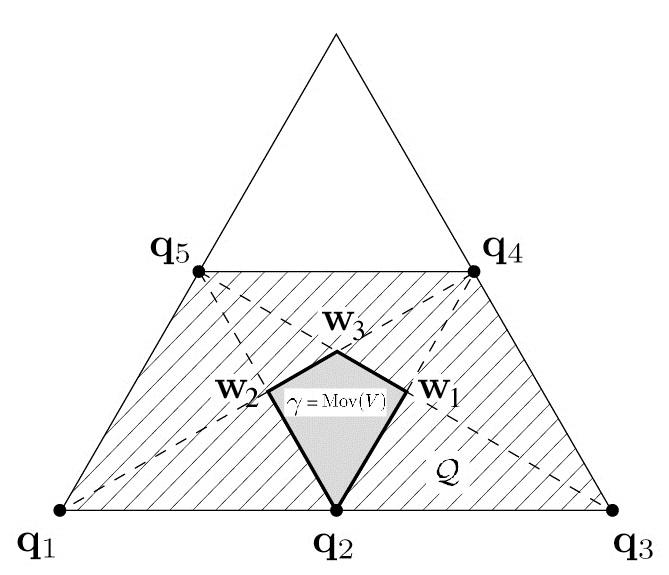}
\caption{\label{fig4}Ex.~\ref{ex:noWPTB}: the section of the cone $\g=\Mov(V)$ inside the Gale dual cone $\mathcal{Q}\subset F^3_+$, as cut out by the plane $\sum_{i=1}^3x_i=1$.}
\end{center}
\end{figure}
 We have then to apply the first condition of Theorem \ref{thm:Q-Fano}(2). Calling $\Si=\Si_{\g}$, we get
        $\mathcal{I}_{\Si}=\{\{1, 2, 4\}, \{2, 4, 5\}, \{1, 3, 4\}, \{1, 3, 5\}, \{2, 3, 5\}\}$
        and
        \begin{eqnarray*}
          \w_{\{1, 2, 4\}} &=& \left(
                                 \begin{array}{c}
                                   2 \\
                                   1 \\
                                   2 \\
                                 \end{array}
                               \right)>0
           \\
          \w_{\{2, 4, 5\}} &=& \left(
                                 \begin{array}{c}
                                   2 \\
                                   1 \\
                                   1 \\
                                 \end{array}
                               \right)>0 \\
          \w_{\{1, 3, 4\}} &=& \left(
                                 \begin{array}{c}
                                   3 \\
                                   1 \\
                                   2 \\
                                 \end{array}
                               \right)>0 \\
          \w_{\{1, 3, 4\}} &=& \left(
                                 \begin{array}{c}
                                   1 \\
                                   3 \\
                                   2 \\
                                 \end{array}
                               \right)>0 \\
          \w_{\{2, 3, 5\}} &=& \left(
                                 \begin{array}{c}
                                   1 \\
                                   2 \\
                                   2 \\
                                 \end{array}
                               \right)>0
        \end{eqnarray*}
        giving that $X(\Si)$ is a Fano toric variety.
\end{example}

\begin{example}\label{ex:WPTB(b)}\cite[Ex.~3.43]{RT-Qfproj} Consider the 4--dimensional PWS of rank 3 given by the following reduced fan and weight matrices
\begin{equation*}
    V=\left(
        \begin{array}{ccccccc}
          1 & 0 & 0 & -1 & 0 & 2 & -4 \\
          0 & 1 & 0 & -1 & 0 & 2 & -4 \\
          0 & 0 & 1 & -1 & 0 & 1 & -2 \\
          0 & 0 & 0 & 0 & 1 & -1 & 1 \\
        \end{array}
      \right)\Rightarrow Q=\left(
                             \begin{array}{ccccccc}
                               1 & 1 & 1 & 1 & 0 & 0 & 0\\
                               0 & 0 & 1 & 2 & 1 & 1 & 0\\
                               0 & 0 & 0 & 0 & 1 & 2 & 1\\
                             \end{array}
                           \right)=\G(V)
\end{equation*}
The first interest of this example is in the fact that
$$\sharp(\P\SF(V))=8<10=\sharp(\SF(V))$$
meaning that $V$ carries two distinct fans of $\Q$--factorial complete toric varieties of rank 3 which are not projective. In particular $X(\Si)$ is singular for every fan $\Si\in\SF(V)$.
\begin{figure}
\begin{center}
\includegraphics[width=7.5truecm]{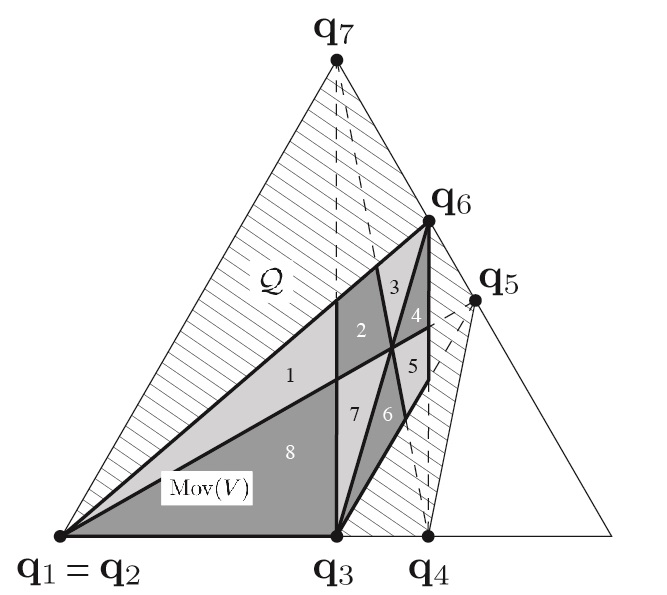}
\caption{\label{fig5}Ex.~\ref{ex:WPTB(b)}: this is the section cut out by the plane $\sum x_i=1$ of the cone $\Mov(V)$, with its eight chambers, inside the Gale dual cone $\mathcal{Q}\subset F^3_+$.}
\end{center}
\end{figure}
Consider the simplicial chamber
    $$\g=\left\langle\begin{array}{ccc}
                  1 & 1 & 1 \\
                  1 & 1 & 2 \\
                  0 & 1 & 2
                \end{array}
    \right\rangle$$
corresponding to the chamber 7 in Fig.~\ref{fig5}. Then
    $$\w_{\Si_{\g}}=\left(
                      \begin{array}{ccc}
                        0 & 1 & -1 \\
                        2 & -2 & 1 \\
                        -1 & 1 & 0 \\
                      \end{array}
                    \right)\cdot Q\cdot\left(
              \begin{array}{c}
                1 \\
              \vdots \\
                1 \\
              \end{array}
            \right)=\left(
              \begin{array}{c}
                1 \\
              2 \\
                1 \\
              \end{array}
            \right)>0
    $$
    meaning that $X(\Si_{\g})$ is a Fano toric variety and that the anti-canonical divisor $-K_X$ is not even a nef divisor for any other fan $\Si\in\SF(V)$ different from $\Si_{\g}$.
    \end{example}

\begin{example}\label{ex:WPTB(c)}\cite[Ex.~3.44]{RT-Qfproj} The present example is obtained by the previous Ex.~\ref{ex:WPTB(b)} by \emph{breaking the symmetry} around the ray $\left\langle(1,2,2)^T\right\rangle\in\Ga(1)$ of the secondary fan $\Ga$. This is enough to get a simplicial and complete fan $\Si\in\SF(V)$ such that $\g_{\Si}=\Nef(X(\Si))=0$. Namely consider the 4--dimensional PWS of rank 3 given by the following reduced fan and weight matrices
\begin{equation*}
    V=\left(
        \begin{array}{ccccccc}
          1 & 0 & -1 & 0 & 0 & 6 & -12 \\
          0 & 1 & -1 & 0 & 0 & 4 & -8 \\
          0 & 0 & 0 & 1 & 0 & -2 & 4 \\
          0 & 0 & 0 & 0 & 1 & -1 & 1 \\
        \end{array}
      \right)\Rightarrow Q=\left(
                             \begin{array}{ccccccc}
                               1 & 1 & 1 & 0 & 0 & 0 & 0\\
                               0 & 2 & 6 & 2 & 1 & 1 & 0\\
                               0 & 0 & 0 & 0 & 1 & 2 & 1\\
                             \end{array}
                           \right)=\G(V)
\end{equation*}
\begin{figure}
\begin{center}
\includegraphics[width=11truecm]{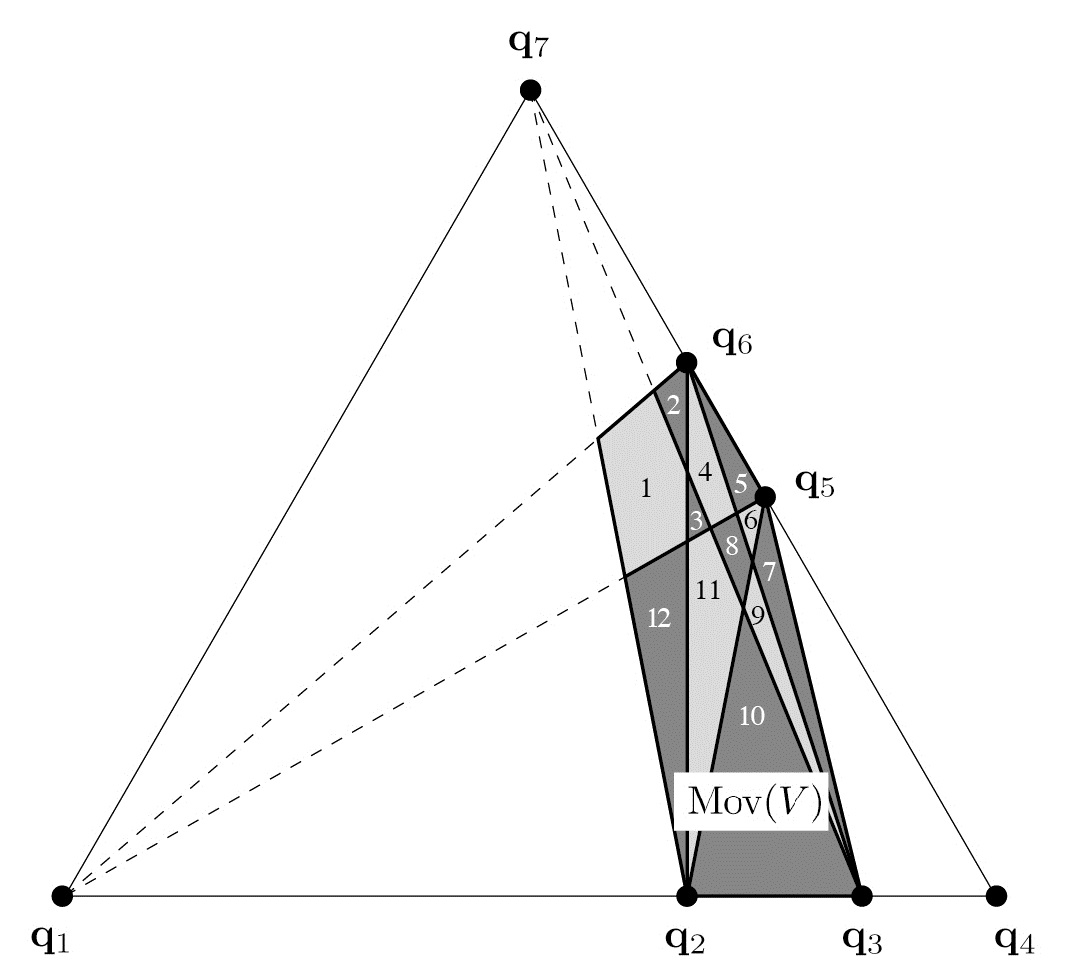}
\caption{\label{fig6}Ex.~\ref{ex:WPTB(c)}: the section of the cone $\Mov(V)$ and its twelve chambers, inside the Gale dual cone $\mathcal{Q}=F^3_+$, as cut out by the plane $\sum_{i=1}^3x_i=1$.}
\end{center}
\end{figure}
   In this case $\sharp(\P\SF(V))=12<13=\sharp(\SF(V))$ and choosing the maxbord and simplicial chamber
    $$\g_{10}:=\left\langle\begin{array}{ccc}
                              1 & 1 & 1 \\
                              2 & 6 & 6 \\
                              0 & 0 & 4
                            \end{array}
        \right\rangle$$
        (see Fig.~\ref{fig6}) the positive multiple of the anti-canonical divisor $-2K_X$ turns out to be ample since
        $$\w_{\Si_{10}}= \left(\begin{array}{ccc}
                                                     0 & 1 & -6 \\
                                                     \frac{1}{3} & -\frac{1}{3} & 1 \\
                                                     -\frac{1}{3} & \frac{1}{3}  & 0
                                                   \end{array}
        \right)\cdot Q\cdot\left(
              \begin{array}{c}
                1 \\
              \vdots \\
                1 \\
              \end{array}
            \right)=\left(
              \begin{array}{c}
                1/2 \\
              3/2 \\
                1 \\
              \end{array}
            \right)>0 $$
        Then $X(\Si_{10})$ is $\Q$--Fano but it is not Gorenstein.

        \noindent From the analysis performed in \cite[Ex.~3.44]{RT-Qfproj}, it turns out that $X(\Si_{10})$ is a toric cover $72:1$ of the weighted projective toric bundle (WPTB) $$\P^W(\mathcal{E})=\P^{(1,2,1)}\left(\mathcal{O}_Y\oplus\mathcal{O}_Y(6D'_4)^{\oplus 2}\right)$$
        where $Y$ is the 2-dimensional PWS of rank 2 whose weight matrix is given by
        $$\widetilde{Q}'=\left(
                            \begin{array}{cccc}
                              1 & 1 & 1 & 0 \\
                              0 & 1 & 3 & 1 \\
                            \end{array}
                          \right)$$
        and $D'_4$ is the torus invariant divisor of $Y$ given by the closure of the torus orbit of the ray generated by the fourth column of $\widetilde{V}'=\G(\widetilde{Q}')$. Then $\P^W(\mathcal{E})$ is not (weak) $\Q$--Fano but it is Gorenstein: in fact $\P^W(\mathcal{E})$ is the toric variety associated with the choice of the chamber
        $$\stackrel{\approx}{\g}=\left\langle\begin{array}{ccc}
                                                   1 & 1 & 3 \\
                                                   1 & 3 & 9 \\
                                                   0 & 0 & 1
                                                 \end{array}
        \right\rangle\subseteq\Mov\left(\stackrel{\approx}{V}\right)$$
        with
        $$
        \stackrel{\approx}{V}=\left(\begin{array}{ccccccc}
          1 & 0 & -1 & 3 & 0 & 0 & 0 \\
          0 & 1 & -1 & 2 & 0 & 0 & 0 \\
          0 & 0 & 0 & 6 & 0 & -1 & 2 \\
          0 & 0 & 0 & 0 & 1 & -1 & 1 \\
        \end{array}
      \right)$$
        and
        $$\w_{\stackrel{\approx}{\Si}}=\left(
                                             \begin{array}{ccc}
                                               \frac{3}{2} & -\frac{1}{2} & 0 \\
                                               0 & 0 & 1 \\
                                               -\frac{1}{2} & \frac{1}{2} & -3 \\
                                             \end{array}
                                           \right)\cdot \stackrel{\approx}{Q}\cdot \left(
              \begin{array}{c}
                1 \\
              \vdots \\
                1 \\
              \end{array}
            \right)=\left(
              \begin{array}{c}
                -4 \\
               4  \\
                -5 \\
              \end{array}
            \right)\,.
        $$
\end{example}

\begin{example}\label{ex:quotient}\cite[Ex.~4.7]{RT-Qfproj}   Consider the following $4\times 7$ $F$--matrix
\begin{equation*}
    V=\left(
        \begin{array}{ccccccc}
          9&11&13&-33&9&44&-97 \\
          10&12&14&-36&10&48&-106 \\
          54&63&75&-192&51&258&-567 \\
          310&365&430&-1105&295&1485&-3265 \\
        \end{array}
      \right)
\end{equation*}
Notice that
\begin{equation*}
  Q=\G(V)=\left(
                                \begin{array}{ccccccc}
                                  1&1&1&1&0&0&0 \\
                                  0&0&1&2&1&1&0 \\
                                  0&0&0&0&1&2&1 \\
                                \end{array}
                              \right)\Rightarrow \G(Q)=
                              \left(
        \begin{array}{ccccccc}
          1 & 0 & 0 & -1 & 0 & 2 & -4 \\
          0 & 1 & 0 & -1 & 0 & 2 & -4 \\
          0 & 0 & 1 & -1 & 0 & 1 & -2 \\
          0 & 0 & 0 & 0 & 1 & -1 & 1 \\
        \end{array}\right)
\end{equation*}
meaning that $X(\Si)$ is a finite abelian quotient of a PWS described in Example~\ref{ex:WPTB(b)}, for every $\Si\in\SF(V)$. Then Corollary~\ref{cor:Q-Fano} allows us to conclude that we get a Fano toric variety for the choice of the chamber
$$\g=\left\langle\begin{array}{ccc}
                  1 & 1 & 1 \\
                  1 & 1 & 2 \\
                  0 & 1 & 2
                \end{array}
    \right\rangle$$
    corresponding to the chamber 7 in Fig.~\ref{fig5}.
\end{example}

\end{document}